\documentclass[12pt]{amsart}

\usepackage[hmargin=2.5cm,bmargin=2.5cm,tmargin=3cm]{geometry}
\usepackage[usenames]{color}
\usepackage{graphbox}

\usepackage{amsmath,amssymb,amsthm,amsfonts,enumerate,url,mathbbol,mathrsfs,esint,tikz,graphicx,pifont,float}
\usepackage{hyperref}

\usepackage{nicefrac}
\usepackage[normalem]{ulem}
\usepackage[utf8]{inputenc}

\theoremstyle{plain}
\newtheorem{theorem}{Theorem}[section]
\newtheorem{lemma}[theorem]{Lemma}
\newtheorem{proposition}[theorem]{Proposition}

\theoremstyle{definition}
\newtheorem{definition}[theorem]{Definition}
\newtheorem{example}[theorem]{Example}
\newtheorem{conjecture}[theorem]{Conjecture}

\newtheorem{assumption}[theorem]{Assumption}
\newtheorem{remark}[theorem]{Remark}

\numberwithin{equation}{section}
\numberwithin{figure}{section}

\newcommand{\dx}{\; {\rm d}x}

  \def\mG{\mathsf{G}} 
    
  \def\mV{\mathsf{V}}
  
  \def\mE{\mathsf{E}}

 \def\mv{\mathsf{v}}
 \def\me{\mathsf{e}}
 \def\mw{\mathsf{w}}
  
  \def\mf{\mathsf{f}}

\newcommand{\R}{\mathbb{R}}
\newcommand{\N}{\mathbb{N}}

\newcommand{\Graph}{\mathcal{G}} 

\newcommand{\VM}{\mathcal{B}} 
\newcommand{\form}{a} 
\newcommand{\Op}{A} 

\DeclareMathOperator{\sign}{sgn} 
\DeclareMathOperator{\acc}{acc} 

\newcommand{\formdomain[1]}{D(#1)} 

\title{On Pleijel's nodal domain theorem for quantum graphs} 

\subjclass[2010]{}


\author[M.~Hofmann]{Matthias Hofmann}
\author[J.~B.~Kennedy]{James B.~Kennedy}
\author[D.~Mugnolo]{Delio Mugnolo}
\author[M.~Pl\"umer]{Marvin Pl\"umer}

\address{Matthias Hofmann, Grupo de F\'isica Matem\'atica, Faculdade de Ci\^encias, Universidade de Lisboa, Campo Grande, Edif\'icio C6, P-1749-016 Lisboa, Portugal}
\email{mhofmann@fc.ul.pt}

\address{James B.~Kennedy, Grupo de F\'isica Matem\'atica \textit{and} Departamento de Matem\'atica, Faculdade de Ci\^encias, Universidade de Lisboa, Campo Grande, Edif\'icio C6, P-1749-016 Lisboa, Portugal}
\email{jbkennedy@fc.ul.pt}

\address{Delio Mugnolo, Lehrgebiet Analysis, Fakult\"at Mathematik und Informatik, Fern\-Universit\"at in Hagen, D-58084 Hagen, Germany}
\email{delio.mugnolo@fernuni-hagen.de}

\address{Marvin Pl\"umer, Lehrgebiet Analysis, Fakult\"at Mathematik und Informatik, Fern\-Universit\"at in Hagen, D-58084 Hagen, Germany}
\email{marvin.pluemer@fernuni-hagen.de}

\date{\today}

\thanks{The work of M.H. and J.B.K. was supported by the Funda\c{c}\~ao para a Ci\^encia e a Tecnologia, Portugal, via the program ``Investigador FCT'', , reference IF/01461/2015 (J.B.K.), and ``Bolseiro de Investigação'', reference
PD/BD/128412/2017 (M.H.), and via project PTDC/MAT-CAL/4334/2014 (M.H. and J.B.K.). The work of D.M. and M.P.\ was supported by the Deutsche Forschungsgemeinschaft (Grant 397230547). The authors would like to acknowledge that this article is based upon work from COST Action CA18232 MAT-DYN-NET, supported by COST (European Cooperation in Science and Technology).}

\begin{document}

\begin{abstract}
We establish metric graph counterparts of Pleijel's theorem on the asymptotics of the number of nodal domains $\nu_n$ of the $n$-th eigenfunction(s) of a broad class of operators on compact metric graphs, including Schrödinger operators with $L^1$-potentials and a variety of vertex conditions as well as the $p$-Laplacian with natural vertex conditions, and without any assumptions on the lengths of the edges, the topology of the graph, or the behaviour of the eigenfunctions at the vertices. {Among other things, these results characterise the accumulation points of the sequence $(\frac{\nu_n}{n})_{n\in\N}$, which are shown always to form a finite subset of $(0,1]$. This} extends the previously known result that $\nu_n\sim n$ \textit{generically}, for certain realisations of the Laplacian, in several directions. In particular, in the special cases of the Laplacian with natural conditions, we show that for graphs with rationally dependent edge lengths, one can find eigenfunctions thereon for which ${\nu_n}\not\sim {n}$; but in this case even the set of points of accumulation may depend on the choice of eigenbasis.
\end{abstract}

\maketitle

\section{Introduction}
Given a {differential} operator with a real eigenfunction $\psi$, its \textit{nodal domains} are the connected components in the support of its positive part along with the connected components in the support of its negative part.
 
The classical Oscillation Theorem, first proved in Sturm's classical paper~\cite{Stu36}, states that the $n$-th eigenfunction $\psi_k$ of a Sturm--Liouville operator with continuous coefficients and separated boundary conditions on a compact interval has $n-1$ zeros in the interior of the interval, that is, $\nu_n=n$ nodal domains. We refer to~\cite{Hin05} for a historical overview of the generalisations of this result, including more general coefficients and boundary conditions.

The counterpart in higher dimensions, Courant's Nodal Domain Theorem~\cite{Cou23}, states that the number $\nu_n$ of nodal domains of the eigenfunction $\psi_n$ associated with the $n$-th eigenvalue of the Dirichlet Laplacian on a bounded domain in $\R^d$ is no larger than $n$. Pleijel's theorem \cite{Ple56}, which establishes an asymptotic bound on the quotient $\nu_n/n$, sharpens Courant's result by stating that the number of eigenvalues for which equality may hold is finite if $d=2$. 

The main goal of the present note is to discuss the behaviour of the sequence $\nu_n/n$, and thus explore the validity, or lack thereof, of Pleijel's theorem, in an intermediate setting between \textit{intervals} on the one hand and higher-dimensional \textit{domains} on the other: compact quantum graphs -- i.e., Laplacians or similar operators on metric graphs with finitely many, finite, edges~\cite{BerKuc13} -- somehow lie in between. This is essentially due to the topological configurations that can be taken by metric graphs.

Over the last century, Courant's theorem has been extended to various settings; quoting~\cite{Ale98}, its proof boils down to three points:\\[-30pt]
{\begin{quote}
\emph{
\begin{enumerate}[(a)]
\item the variational characterisation of eigenvalues,
\item the maximum principle,
\item the Unique Continuation Property. 
\end{enumerate}
}
\end{quote}}

And this is where troubles arise. Indeed, the first two points (a), (b) can be enforced by assuming the relevant operator to be associated with a quadratic form that satisfies the first Beurling--Deny condition and whose domain is compactly embedded in the Hilbert space: this includes e.g.\ the cases of Laplacians with Neumann or Robin boundary conditions; of discrete Laplacians on combinatorial graphs; or Laplacians on metric graphs with natural vertex conditions (i.e., continuity on the whole metric graph along with a Kirchhoff condition on the normal derivatives at each vertex); this is the setting of the most general currently available versions of Courant's theorem~\cite{KelSch20}. 

But the third one (c) has a geometric flavour and is known to fail -- in general -- for both combinatorial and metric graphs. 

Given an eigenfunction $\psi$, failure of the Unique Continuation Property on a (combinatorial or metric) graph implies that classical nodal domains may not exhaust the graph: this suggests to study both \textit{strong} and \textit{weak} nodal domains, i.e., the (closures of the) connected components of both $\{\psi>0\}$ and $\{\psi\ge 0\}$, and likewise of both $\{\psi<0\}$ and $\{\psi\le 0\}$. Courant-type bounds on the nodal count of eigenvectors associated with the $n$-th (possibly non-simple) eigenvalue of a combinatorial graph were derived in the seminal paper~\cite{DavGlaLey01}; see also~\cite{Urs18} for an overview and refinement of later results inspired by~\cite{DavGlaLey01}. A further complication, both on combinatorial and metric graphs, arises from the possibility that $\nu_n$ may depend on the specific choice of eigenbasis if the eigenvalues are not simple.

However, on metric graphs we still have the property that if an eigenfunction vanishes on an open subset of an edge, then it must vanish on the whole edge; so a way to enforce the Unique Continuation Property, and also to remove any ambiguities regarding non-simple eigenvalues, is simply to assume that all eigenvalues are simple and no eigenfunctions vanish on any vertices, which is known to be true for the usual realisations of the Laplacian under certain topological assumptions on the graph (no loops) and then genericity assumptions on the edge lengths. Under these assumptions, it is known that equality of $\nu_n$ and $n$ holds for metric trees (\cite{PokPryAlo96}) and only for them (\cite{Ban14}).

In the case of quantum graphs which may have cycles, the Courant--Pleijel theory was first obtained -- again, only under the aforementioned genericity assumptions -- by Gnutzmann, Smilansky and Weber in~\cite{GnuSmiWeb04}; their proof mirrors the original one by Pleijel but, unlike in Pleijel's result, it only yields that the number $\nu_n$ of nodal domains  associated with the $n$-th eigenfunction is \textit{generically} bounded from above by $n$. Under the same assumptions, the nodal deficiency $n-\nu_n$ has since been studied by Band, Berkolaiko and their co-authors in several papers since~\cite{BanBerRaz12} {(see, e.g., \cite{ABB18,BerWey14} and the references therein)}.

One of the goals of this note is to describe the asymptotic behaviour of the sequence \(\nu_n\) as \(n\rightarrow\infty\), removing this genericity assumption, for any given choice of orthonormal basis of eigenfunctions, where far richer behaviour is possible. Here we propose a refinement that uses an elementary isoperimetric inequality due to Nicaise~\cite{Nic87} {which} seems to have been little known at the time of~\cite{GnuSmiWeb04}. This way we can drop any condition on the eigenfunctions, {including the usual genericity conditions,} which seem hard to check on graphs with complicated topologies and typically fail on graphs with non-trivial automorphism group. More interestingly, we present classes of graphs (including all graphs whose shortest edge is a loop) for which the number of nodal domains -- counted along a suitable sequence of mutually orthogonal eigenfunctions --  is actually strictly smaller than the index of the corresponding eigenvalue, {and remains so in the asymptotic limit}. We will also demonstrate the strength of our approach by applying it to two classes of operators not previously considered in this context: quite general Schrödinger operators with (non-negative) $L^1$-potentials and a large variety of vertex conditions, and the $p$-Laplacian with natural vertex conditions.

Let us now describe our main results, and the structure of the paper, more precisely. After recalling the standard definition of metric graphs and introducing the general class of Schrödinger operator we are going to work with in Section~\ref{sec:gen-setting}, our first main results, Theorem~\ref{thm:mainresult} and Proposition~\ref{prop:accumulation}, are presented in Section~\ref{sec:plei-schr}; these describe the set accumulation points of the sequence $(\frac{\nu_n}{n})_{n\in\mathbb N}$ for the aforementioned Schrödinger operators; in particular, these points of accumulation always form a finite subset of $(0,1]$, and can be described explicitly in terms of the edge lengths of the graph. The proofs rely upon a series of lemmata that are discussed in Subsection~\ref{sec:proofs}.

Our results can be significantly refined if we restrict to the case of Laplacians with natural vertex conditions. We delve into this setting in Section~\ref{sec:plei-lapl-nat}; see Theorem~\ref{thm:standard} in particular, where among other things we show that on graphs with rationally dependent edge lengths $(\frac{\nu_n}{n})_{n\in\mathbb N}$ will always have at least two points of accumulation, one strictly smaller than $1$. In this section we also give a simple example, {Example~\ref{ex:4-star-choice}} demonstrating that even the set of points of accumulation of $(\frac{\nu_n}{n})_{n\in\N}$ itself may actually depend on the choice of eigenbasis, even for the Laplacian with natural conditions, reinforcing the principle that many features of interest are lost if we restrict to the ``generic'' setting where all eigenvalues are simple.

With the aim of showing the flexibility of our approach -- which, unlike that of~\cite{GnuSmiWeb04} does not rely on global linear algebraic manipulations, but rather on isoperimetric inequalities applied locally to the nodal domains -- in the last part of the paper we turn to an important \textit{nonlinear} operator. Section~\ref{sec:pleijel-p} is devoted to the theory of $p$-Laplacian, and to obtaining a Pleijel-type theorem in this context, Theorem~\ref{thm:pleijel-p}. On metric graphs, these nonlinear operators were introduced in~\cite[Section 6.7]{Mug14} and their spectral properties were studied in \cite{DelRos16,BerKenKur17}; their theory is significantly less developed than in the case of intervals, though. In turn, the theory of general $p$-Laplace operators is not as well understood as in the linear case of $p=2$: even the existence of infinitely many eigenvalues or the validity of the Unique Continuation Property seem to be unknown in general. A Courant-type Nodal Domain Theorem on domains was proved in~\cite{DraRob02}, though; it yields an upper bound $\nu_{n,p}\le 2n-2$, which can be refined to $\nu_{n,p}\le n$ under the additional assumption that the unique continuation principle prevails. Results {of Sturm Oscillation type} for $p$-Laplacians with potential in one dimension were obtained in~\cite{ReiWal99,BinDra03}. It is all the more interesting that our Theorem~\ref{thm:pleijel-p} implies in particular the sharper bound $\nu_{n,p}\le n$ even in an environment where the Unique Continuation Property clearly fails.

We include a number of auxiliary results about the $p$-Laplacian on metric graphs and its eigenvalues in Appendix~\ref{sec:p-weyl}, where, in particular, we give the result that its variational eigenvalues satisfy the same Weyl asymptotics as the $p$-Laplacian on an interval. Finally, in Appendix~\ref{app:lambda1-est}, we give a bound on the first eigenvalue of a general (linear) Schrödinger operator in terms of the $L^1$-norm of its potential and the average edge length of $\Graph$ only, reminiscent of \cite{KenKurMal16}, and needed for the proofs in Section~\ref{sec:plei-schr}; despite being rather elementary, it may be of some independent interest.

\section{General setting}\label{sec:gen-setting}

Let \(\mathcal G\) be a \textit{compact metric graph}, i.e., a finite combinatorial graph $\mG=(\mV,\mE)$ each of whose edges $\me\in\mE$ is identified with an interval $(0,\ell_\me)$ of finite length; we write $|\mathcal G| = \sum_{\me \in \mE} \ell_\me$ for the total length of the graph. See~\cite{Mug19} for a precise definition of this metric measure space and the function spaces $C(\mathcal G)$ and $L^2(\mathcal G)$, as well as the Sobolev space $H^1(\mathcal G)$ of all continuous functions over $\mathcal G$ that are edgewise weakly differentiable with weak derivative in $L^2(\mathcal G)$.

In what follows we will give a description of the operators we will be considering in what follows. Note that all we will need for the results there are certain more or less abstract properties which these operators satisfy; in particular, the reader unfamiliar with the theory presented in this section may imagine the important special case of a Schrödinger operator with smooth (or even zero) potential and (possibly) delta couplings, or else any of the usual vertex conditions, at the vertices.

We first consider a possible relaxation of the continuity condition at the vertices to allow for \textit{weighted} continuity encoded in a nonnegative vector of edge weights $w_\mv\in \R^{\deg(\mv)}_+$, $\mv\in\mV$, i.e.,
\begin{equation}\label{eq:wei-cnd}
w_{\me,\mv} f_\me(\mv)=w_{\mf,\mv} f_\mf(\mv)\quad \hbox{if }\mv\in \me\cap \mf .
\end{equation}
Indeed, in this case we can define, in a natural way, a space $H^1_w(\mathcal G)$ of edgewise $H^1$-functions that satisfy~\eqref{eq:wei-cnd} at the vertices and repeat the above reasoning for the operator $\Op (q,\VM ,w)$ thus arising. Note that while functions in $H^1_w(\mathcal G)$ may be discontinuous at the vertices, they can only change sign at a vertex, i.e., take on positive and negative values in any neighbourhood of a vertex, if they are zero at that vertex.

We then define, for $q\in L^1(\mathcal G)$ and a family of matrices $\VM \in M_{2|\mE|\times 2|\mE|}(\mathbb C)$ the sesquilinear form
\begin{equation}\label{eq:form-2}
\begin{split}
\form (f) &:=\int_{\mathcal G} \left(|f'(x)|^2 + q(x)|f(x)|^2\right) \textrm{d}x + \sum_{\mv,\mw\in\mV} \left(\VM_{\mv\mw} f(\mw),f (\mv)\right)_{\mathbb C^{\deg(\mv)}},\\
\formdomain[\form]&:=H^1_w(\mathcal G),
\end{split}
\end{equation}
where for each $\mv,\mw\in \mV$, $\VM_{\mv\mw}$ is a $\deg(\mv)\times \deg(\mw)$-matrix, and for \(\mv\in\mV\), \(f(\mv)=(f_\me(\mv))_{\me\in\mE_\mv}\). If we want to emphasise the dependence on the potential and the vertex conditions, then we will also write $\form_{q,\VM,w}$ in place of just $\form$. At any rate, it follows from the theory presented in \cite[Section~6.5]{Mug14} that this form is bounded and elliptic; hence the associated operator $\Op = \Op (q,\VM,w)$ is (minus) the generator of an analytic, strongly continuous semigroup on $L^2(\mathcal G)$. This semigroup is of trace class and therefore $\Op$ has pure point spectrum.

If in particular $q$ is real-valued and $\VM$ is Hermitian for all $\mv\in\mV$, then $\form$ is a closed quadratic form, hence $\Op(q,\VM,w)$ is a self-adjoint operator that is bounded from below. This setting includes, as special cases, realisations of the Laplacian on $\mathcal G$ with so-called \textit{natural vertex conditions} (continuity across vertices, all normal derivatives sum up to 0 at each vertex), corresponding to $q \equiv 0$, $\VM = 0$ and $w \equiv 1$, as well as (standard) delta couplings (continuity across vertices, at each vertex the sum of all normal derivatives equals minus the point evaluation at the same vertex), where $q \equiv 0$, $w \equiv 1$ and $\VM$ is a diagonal matrix.
 
Now it is known, cf.~\cite[Theorem~6.71]{Mug14}, that the semigroup is positive if and only if so is the semigroup generated by each $-\VM$ (this is in particular the case if $\VM$ is diagonal, which covers delta couplings, including weighted versions thereof). In this context, we refer to the condition \eqref{eq:wei-cnd} and the weighted Kirchhoff--Robin-type condition associated with the matrix $\VM$ collectively as \textit{positivity preserving} vertex conditions. Finally, all these assertions remain valid if, for some $\mV_0\subset \mV$ (where possibly, trivially, $\mV_0 = \emptyset$), we consider the operator $\Op(q,\VM,w,\mV_0)$ associated with the restriction of the form $a$ to $H^1_{0,w}(\mathcal G;\mV_0)$, the space of all functions in  $H^1_w(\mathcal G)$ that vanish on the vertices in $\mV_0$ {(in this case, of course, we only require that $\VM$ be defined for $\mv \in \mV \setminus \mV_0$).}

The Schrödinger operators associated with these classes of forms were thoroughly studied in~\cite{Kur19}.

In all these cases, the discrete spectrum of $\Op(q,\VM,w,\mV_0)$ consists of real eigenvalues $\lambda_n (q, A, w,\mV_0)$ repeated according to their finite multiplicities, characterised by the usual Courant--Fischer max-min and min-max principles, which diverge to $+\infty$ as $n \to \infty$, and whose eigenfunctions may be chosen to be real and to form an orthonormal basis $(\psi_n)_{n\in\mathbb N}$ of $L^2(\mathcal G)$. (We mostly avoid this heavy notation and simply write $\lambda_n(\mathcal G):= \lambda_n (q, \VM, w,\mV_0)$.)

If additionally $\mathcal G$ is connected, $\VM$ is diagonal, and for a given (and possibly empty) $\mV_0\subset \mV$, the set $\Graph \setminus \mV_0$ is still connected, then the semigroup generated by $\Op(q,\VM,w,\mV_0)$ is even irreducible (see~\cite[Proposition~3.7]{Mug07} for the special case of $q\equiv1$; the proof is identical in the general case), hence by the Kre\u{\i}n--Rutman Theorem we deduce that the first eigenspace is one-dimensional and spanned by a positive function (the Perron eigenfunction) $\psi_1$: i.e., $\psi_1(x)>0$ for a.e.\ $x\in \mathcal G$. Indeed, more holds: it was proved in~\cite{Kur19} that a strong maximum principle holds, namely the Perron eigenfunction vanishes only at the vertices in $\mV_0$.
In particular, the form domain contains the set 
\begin{equation}\label{eq:h10}
H^1_0 (\mathcal G) := H^1_0(\mathcal G; \mV)
\end{equation} 
of globally $H^1$-functions which vanish at \emph{all} vertices.

We will denote by $\lambda_n^D = \lambda_n^D (\mathcal G)$ the $n$-th lowest eigenvalue (counting multiplicities) of the Schrödinger operator with potential $q$ and Dirichlet conditions at \emph{all} vertices of $\mathcal G$, that is, whose form domain is $H^1_0(\mathcal G)$; in this case the graph decomposes into a disjoint collection of intervals, moreover, the associated sesquilinear form is exactly \eqref{eq:form-2} restricted to $H^1_0 (\mathcal G)$. We note the following eigenvalue interlacing result for future reference.

\begin{lemma}
\label{lem:interlacing}
With the above assumptions and notation, for all $n \geq |\mV| + 1$ we have
\begin{displaymath}
	\lambda_{n-|\mV|}^D (\mathcal G) \leq \lambda_n (\mathcal G) \leq \lambda_n^D (\mathcal G).
\end{displaymath}
\end{lemma}

\begin{proof}
Both inequalities are an immediate consequence of the min-max characterisation of the respective eigenvalues and the fact that the forms agree on $H^1_0 (\mathcal G)$, the latter in conjunction with the inclusion of the form domains $H^1_0 (\mathcal G) \subset \formdomain[\form]$, the former in conjunction with the fact that the quotient space $\formdomain[\form] / H^1_0 (\mathcal G)$ is at most $|\mV|$-dimensional (cf.\ also \cite[Section~4.1]{BeKeKuMu19}).
\end{proof}

\section{Pleijel's theorem for Schrödinger operators on metric graphs}\label{sec:plei-schr}

Our main result is a variation of Pleijel's theorem for metric graphs. We impose the following assumptions throughout this section.

\begin{assumption}\label{assum:basic}
$\mathcal G$ is a compact, connected metric graph with underlying combinatorial graph $\mG=(\mV,\mE)$ and edge lengths $\ell_\me$, $\me\in\mE$; we set $\ell_{\min} := \min_{\me \in \mE} \ell_\me$. We also fix a (possibly empty) set $\mV_0\subset \mV$ and a potential $0 \leq q\in L^1(\mathcal G)$, and suppose $\VM$ is a Hermitian $2|\mE|\times 2|\mE|$-matrix such that the semigroup $(e^{-t\VM})_{t\ge 0}$ is positive and $(w_\mv)_{\mv\in\mV}\in \R^{2|\mE|}$ is a vector such that $w_\mv\in {\R}^{\deg(\mv)}_+$ for all {$\mv\in\mV \setminus \mV_0$.}
\end{assumption}

{We recall that $(e^{-t\VM})_{t\ge 0}$ is positive if and only if all entries of $\VM$ are real and all off-diagonal entries are non-positive; this includes all (real) delta coupling conditions.}

Under these assumptions, we will consider the operator associated with the form $\form_{q,\Op,w}$ introduced in Section~\ref{sec:gen-setting}. In this section we fix once and for all an (\textit{a priori} arbitrary) eigenbasis of this operator.

\begin{definition}
Let $(\psi_n)_{n\in\mathbb N}$ be an orthonormal sequence of eigenfunctions $(\psi_n)_{n\in\N}$ with associated eigenvalues $(\lambda_n)_{n\in\N}$ of the Schrödinger operator $\Op(q,\VM,w,\mV_0)$ associated with the form $\form_{q,\VM,w}$. As already mentioned in the introduction, the \textit{nodal domains} of any eigenfunction $\psi_k$ are the respective closures in the metric space $\mathcal G$ of connected components of the sets $\{\psi_k \neq 0 \}$. We occasionally denote by $\mathcal G_1,\ldots,\mathcal G_{\nu_k}$ the nodal domains themselves, and by $\partial \mathcal G_i$ the boundary of $\mathcal G_i$ in $\mathcal G$, i.e., 
\[
\partial \mathcal G_i:=\overline{\mathcal G_i}\cap \bigcup_{j\ne i}\overline{\mathcal G_i}.
\] 
We denote the nodal count of this sequence by $(\nu_n)_{n\in\N}$.
\end{definition}

\textit{A priori} the sequence \(\nu_k\in\N\), including the points of accumulation, can depend on the precise choice of basis, see Example~\ref{ex:4-star-choice} below, unless suitable assumptions on the edge lengths $(\ell_\me)_{\me\in\mE}$ and the graph topology are imposed that force all eigenvalues to be simple. 

Furthermore, here and throughout, given a sequence $(a_n)_{n\in\mathbb N} \subset \R$, we will write
	\[\acc\left\{a_n : n\in\N\right\}\]
to denote its set of points of accumulation. With this we are now ready to formulate our first main theorem.

\begin{theorem}
\label{thm:mainresult}
For all quantum graphs satisfying Assumption~\ref{assum:basic}, the nodal count $(\nu_n)_{n\in\N}$ satisfies
\begin{equation}
\label{eq:accumulation}
	\acc \left\{ \tfrac{\nu_n}{n} : n \in \N \right\} \subset \left\{ \frac{\sum_{\me \in \mE_0} \ell_\me}{|\mathcal G|} : \mE \supset \mE_0
	\text{ is a nonempty set of edges} \right\}.
\end{equation}
In particular, $\acc \left\{ \tfrac{\nu_n}{n} : n \in \N \right\}$ is a finite set, and
\begin{equation}
\label{eq:estimates}
	0 < \frac{\ell_{\min} }{|\mathcal{G}|} \leq \liminf_{n \to \infty} \frac{\nu_n}{n} \leq \limsup_{n \to \infty} \frac{\nu_n}{n} \leq 1.
\end{equation}
\end{theorem}

While the right-hand side of \eqref{eq:accumulation} does not depend on the parameters $q,\VM,w$, the set inclusion in \eqref{eq:accumulation} is sharp in the case of a graph consisting of just one interval. Indeed, recall that on intervals, in the case of Sturm--Liouville problems, $\nu_n = n$ for all $n \in \N$ (see~\cite{Hin05}).

As mentioned in the introduction, the key driving force behind the potential appearance of a non-trivial set of points of accumulation of $\frac{\nu_n}{n}$ between $0$ and $1$ here is the failure of the unique continuation principle, as evidenced by the following characterisation.

\begin{proposition}
\label{prop:accumulation}
Under Assumption~\ref{assum:basic} we have
\begin{equation}
\label{eq:accumulation-eigenfunctions}
	\acc \left\{ \tfrac{\nu_n}{n} : n \in \N \right\} = \acc \left\{ \frac{|\{\psi_n \neq 0\}|}{|\mathcal G|} : n \in \N \right\}.
\end{equation}
\end{proposition}

Note in particular that the points of accumulation of $\frac{\nu_n}{n}$ are invariant under uniform rescaling of all edges of the graph. 

 The proof of Theorem~\ref{thm:mainresult} and Proposition~\ref{prop:accumulation} is based on the following principles, the proofs of which, in turn, are postponed to Subsection~\ref{sec:proofs}.

\begin{lemma}[Weyl asymptotics]
\label{lem:weyl}
We have
\begin{equation}
\label{eq:weyl}
	\lambda_n (\mathcal G) = \frac{\pi^2}{|\mathcal G|^2} n^2 + o(n^2) \qquad \text{as } n \to \infty.
\end{equation}
\end{lemma}

\begin{lemma}[Relationship between $\nu_n$ and $\lambda_n$]
\label{lem:nu-lambda}
There exists $n_0 \in \N$ depending only on the metric graph $\mathcal G$ and the potential $0 \leq q \in L^1(\Graph)$ such that, for all $n \geq n_0$,
\begin{equation}
\label{eq:nu-lambda-estimates}
	|\{\psi_n \neq 0\}| \cdot \frac{\lambda_n (\mathcal G)^{1/2}-\|q\|_1}{\pi} - (2|\mE|-1)|\mV|
	\leq \nu_n \leq |\{\psi_n \neq 0\}| \cdot \frac{\lambda_n (\mathcal G)^{1/2}}{\pi} + |\mV|.
\end{equation}
In particular,
\begin{equation}
\label{eq:nu-lambda-asymptotics}
	\nu_n = |\{\psi_n \neq 0\}| \cdot \frac{\lambda_n (\mathcal G)^{1/2}}{\pi} + O(1) \qquad \text{as } \lambda_n (\mathcal G) \to \infty.
\end{equation}
\end{lemma}

These two lemmata are logically independent of each other; in particular, in \eqref{eq:nu-lambda-estimates} we explicitly do not use the Weyl asymptotics to estimate $\lambda_n$. Lemma~\ref{lem:weyl} in particular can be refined significantly for specific types of vertex conditions and potentials; for example, in the case of the Laplacian with natural vertex conditions and if $\mathcal G$ is not a cycle, then we may strengthen \eqref{eq:weyl} to
\begin{displaymath}
	\left(n - \frac{|N|+\beta}{2}\right)^2 \frac{\pi^2}{|\mathcal G|^2} \leq \lambda_n (\mathcal G) \leq \left(n-2+\beta + \frac{|N| + \beta}{2}\right)^2 \frac{\pi^2}{|\mathcal G|},
\end{displaymath}
where $|N|$ is the number of degree one vertices and $\beta$ is the first Betti number (number of independent cycles) of the graph, as follows from \cite[Theorems~4.7 and~4.9]{BerKenKur17}. More generally, if $q \in L^\infty (\mathcal G)$, then we may obtain the two-sided estimate
\begin{equation}
\label{eq:better-weyl}
	(n-c_1)^2 \frac{\pi^2}{|\mathcal G|^2} \leq \lambda_n (\mathcal G) \leq (n+c_2)^2 \frac{\pi^2}{|\mathcal G|} \qquad \text{for all } n \in \N
\end{equation}
for constants $c_1,c_2>0$ depending only on $\mathcal G$ and $\|q\|_\infty$; this is a consequence of Lemma~\ref{lem:interlacing} and a simple variational argument bounding $q$ in terms of the constant potential $\|q\|_\infty$, and zero.

Let us now show how Lemmata~\ref{lem:weyl} and~\ref{lem:nu-lambda} lead to the proofs of the main results. To prove Proposition~\ref{prop:accumulation}, if we combine \eqref{eq:nu-lambda-asymptotics} and \eqref{eq:weyl}, then we obtain the asymptotic behaviour
\begin{equation}
\label{eq:accumulation-eigenfunctions-o}
	\nu_n = \frac{|\{\psi_n \neq 0\}|}{|\mathcal G|} n + o(n) \qquad \text{as } n \to \infty.
\end{equation}
The other ingredient in the proof of Theorem~\ref{thm:mainresult} is the following ``weak'' unique continuation principle, whose proof will also be given in Subsection~\ref{sec:proofs}; Theorem~\ref{thm:mainresult} is a direct consequence of \eqref{eq:accumulation-eigenfunctions-o} and \eqref{eq:values-of-psin}.

\begin{lemma}[Possible values of $|\{\psi_n \neq 0\}|$]
\label{lem:values-of-psin}
For any vertex conditions under consideration and all $n \in \N$, we have
\begin{equation}
\label{eq:values-of-psin}
	\Big\{ |\{\psi_n \neq 0\}| : n \in \N \Big\} \subset \left\{ \sum_{\me \in \mE_0} \ell_\me : \mE_0 \subset \mE 
	\text{ is a nonempty set of edges} \right\}.
\end{equation}
\end{lemma}

\subsection{Proofs of the lemmata}\label{sec:proofs}
Here we give the proofs of the three lemmata which combined yield Theorem~\ref{thm:mainresult}. We suppose throughout, without further comment, that Assumption~\ref{assum:basic} holds.

\begin{proof}[Proof of Lemma~\ref{lem:weyl}]
This is an immediate consequence of Lemma~\ref{lem:interlacing}, together with the fact that the eigenvalues of the operator associated with the restriction of the form $\form (q,\VM,w)$ to $H^1_0 (\mathcal G)$, that is, with Dirichlet boundary conditions everywhere, satisfy the usual Weyl asymptotics on any bounded interval and thus any finite union of disjoint intervals, see, e.g., \cite[Lemma~2.1]{AtkMin87}.
\end{proof}

The second, Lemma~\ref{lem:nu-lambda}, is in turn based on the principle that $\lambda_n (\mathcal G)$ is always the first eigenvalue of any nodal domain of $\psi_n$, and as a consequence, that the maximal size of any nodal domain converges to zero as $n \to \infty$.

\begin{lemma}
\label{lem:first-eigenvalue}
Given $n \in \N$, the eigenvalue $\lambda_n (\mathcal G)$, and the associated eigenfunction $\psi_n$, with nodal domains $\mathcal G_1, \ldots, \mathcal G_{\nu_n}$, for each $j=1,\ldots,\nu_n$ we have
\begin{displaymath}
	\lambda_n (\mathcal G) = \lambda_1 (\mathcal G_j),
\end{displaymath}
where the operator associated with the latter eigenvalue has Dirichlet conditions at all the boundary points of $\mathcal G_j$ corresponding to zeros of $\psi_n$ (but the same vertex conditions as before at the interior vertices of $\mathcal G_j$, and the same potential $q$ restricted to $\mathcal G_j$).
\end{lemma}

\begin{proof}
Suppose without loss of generality that $\psi_n \geq 0$ in $\mathcal G_j$, with strict inequality except at the Dirichlet vertices of $\mathcal G_j$, and set $\phi_n := \psi_n \chi_{\mathcal G_j}$; in a slight abuse of notation, we will identify $\phi_n$ with its restriction to $\mathcal G_j$ in $L^2 (\mathcal G_j)$. We observe that $\phi_n$ is a classical solution on $\mathcal G_j$, as by construction it satisfies the equation edgewise and the vertex conditions classically; moreover, the corresponding eigenvalue, which we can read off the eigenvalue equation, is $\lambda_n (\mathcal G)$.

That is, $\lambda_n (\mathcal G)$ is an eigenvalue of $\mathcal G_j$, i.e., $\lambda_n (\mathcal G) = \lambda_k (\mathcal G_j)$ for some $k \geq 1$; moreover, its eigenfunction $\phi_n$ is, by construction, strictly positive in $\mathcal G_j$ except at boundary points of $\mathcal G_j$ and any interior Dirichlet vertices. By \cite[Theorem~3]{Kur19}, it is possible to choose the first eigenfunction $\varphi_1$ of $\lambda_1 (\mathcal G_j)$ to have this property. Orthogonality of eigenfunctions on $\mathcal G_j$ belonging to different eigenspaces implies that $\phi_n = c\varphi_1$ for some $c>0$, and we conclude that the corresponding eigenvalues must be equal, $\lambda_n (\mathcal G) = \lambda_1 (\mathcal G_j)$.
\end{proof}

The other ingredient we need for the proof of Lemma~\ref{lem:nu-lambda} is an estimate on the first eigenvalue of any operator $\Op(q,\VM,w,\mV_0)$ on any compact, connected graph (which in practice will be one of the nodal domains of $\psi_n$), which is given in Proposition~\ref{prop:lambda1-est} in Appendix~\ref{app:lambda1-est}. This proposition, when applied to the nodal domains $\Graph_j$ of $\psi_n$ upon invoking Lemma~\ref{lem:first-eigenvalue}, leads to the following estimate on the size of $\Graph_j$.

\begin{lemma}
\label{lem:nodal-size-est}
For all $n \in \N$, for all nodal domains $\mathcal G_j$, $j=1,\ldots,\nu_n$, we have
\begin{equation}
\label{eq:nodal-size-est}
	|\mathcal G_j| \leq \frac{2\pi |\mE|}{\sqrt{\lambda_n (\Graph) + \|q\|_1^2} - \|q\|_1}.
\end{equation}
In particular, if $\lambda_n (\mathcal G)$ is sufficiently large; explicitly, if
\begin{displaymath}
	\lambda_n (\mathcal G) > \left(\frac{2\pi |\mE|}{\ell_{\min}} + \|q\|_1\right)^2 - \|q\|_1^2,
\end{displaymath}
then no nodal domain can contain more than one vertex of $\Graph$.
\end{lemma}

\begin{proof}
Fix a nodal domain $\mathcal G_j$; then $\mathcal G_j$ certainly cannot have more than $2|\mE|$ edges (note that it could contain both ends of a given edge in $\mathcal G$ without containing the whole edge). Now by Lemma~\ref{lem:first-eigenvalue}, we have $\lambda_n (\mathcal G) = \lambda_1 (\mathcal G_j)$; combining this with the estimate \eqref{eq:pseudoKKMM} applied to $\mathcal G_j$ yields
\begin{displaymath}
	\lambda_n (\mathcal G) \leq \left(\frac{2\pi |\mE|}{|\Graph_j|}+\|q\|_1\right)^2 - \|q\|_1^2.
\end{displaymath}
Rearranging yields \eqref{eq:nodal-size-est}. If $\lambda_n (\mathcal G)$ is sufficiently large as stated, then $|\mathcal G_j| < \ell_{\min}$ for all $j$, meaning no nodal domain can contain an entire edge.
\end{proof}

\begin{remark}
\label{rem:nodal-size-est}
The proof shows that if $\mathcal G_j$ is an interval, then, since we may take $|\mE (\mathcal G_j)|=1$ in \eqref{eq:pseudoKKMM}, \eqref{eq:nodal-size-est} may be improved to
\begin{equation}
\label{eq:better-nodal-size-est}
	|\mathcal G_j| \leq \frac{\pi}{\sqrt{\lambda_n (\Graph) + \|q\|_1^2} - \|q\|_1}.
\end{equation}
\end{remark}

\begin{proof}[Proof of Lemma~\ref{lem:nu-lambda}]
Firstly observe that by definition of the nodal domains, for any $n \in \N$, we have
\begin{equation}
\label{eq:nodal-domain-sum}
	|\{\psi_n \neq 0\}| = \sum_{j=1}^{\nu_n} |\mathcal G_j|.
\end{equation}
Now note that $\lambda_n (\mathcal G) \to \infty$ (this follows from the compactness of the resolvent and the semi-boundedness of the form $\form_{q,\VM,w}$, but can also be obtained as a consequence of Lemma~\ref{lem:weyl}). Hence, by Lemma~\ref{lem:nodal-size-est} there exists some $n_0 \in \N$, which may be chosen to depend only on the metric graph $\mathcal G$ and $q$, such that each nodal domain $\mathcal G_j$ contains at most one vertex of $\mathcal G$, for all $n \geq n_0$. For such $n$, we suppose the nodal domains are ordered in such a way that $\mathcal G_1,\ldots, \mathcal G_{|\mV|}$ each contain at most one vertex, while $\mathcal G_{|\mV|+1},\ldots,\mathcal G_{\nu_n}$ are all intervals; in particular, for all $j \geq |\mV|+1$, by Lemma~\ref{lem:first-eigenvalue}, $\lambda_n (\mathcal G) = \lambda_1 (\mathcal G_j)$. Now on the one hand, since $q \geq 0$, $\lambda_1 (\mathcal G_j) \geq \pi^2/|\mathcal G_j|^2$, whence
\begin{displaymath}
	|\mathcal G_j| \geq \frac{\pi}{\lambda_n (\mathcal G)^{1/2}}.
\end{displaymath}
On the other hand, using \eqref{eq:better-nodal-size-est}, for such nodal domains we also have, supposing without loss of generality that $\lambda_{n_0} > \|q\|_1^2$,
\begin{displaymath}
	|\mathcal G_j| \leq \frac{\pi}{\sqrt{\lambda_n (\Graph) + \|q\|_1^2} - \|q\|_1}
	\leq \frac{\pi}{\lambda_n(\mathcal G)^{1/2} - \|q\|_1}.
\end{displaymath}
Summing over $j$, we obtain the two-sided estimate
\begin{equation}
\label{eq:nu-lambda-intermediate}
	(\nu_n - |\mV|) \cdot \frac{\pi}{\lambda_n (\mathcal G)^{1/2}}+ \sum_{j=1}^{|\mV|} |\mathcal G_j|
	\leq |\{\psi_n \neq 0\}| \leq
	(\nu_n - |\mV|) \cdot \frac{\pi}{\lambda_n(\mathcal G)^{1/2} - \|q\|_1}+ \sum_{j=1}^{|\mV|} |\mathcal G_j|.
\end{equation}
Invoking \eqref{eq:nodal-size-est}, we may estimate the size of the first $|\mV|$ nodal domains by
\begin{displaymath}
	0 \leq \sum_{j=1}^{|\mV|} |\mathcal G_j| \leq \frac{2\pi |\mE| |\mV|}{\lambda_n (\mathcal G)^{1/2}-\|q\|_1}.
\end{displaymath}
Using this in \eqref{eq:nu-lambda-intermediate} and rearranging yields \eqref{eq:nu-lambda-estimates}.
\end{proof}

\begin{proof}[Proof of Lemma~\ref{lem:values-of-psin}]
It suffices to prove the following unique continuation statement: if any eigenfunction $\psi_n$ has a zero at some point $x$ in the interior of an edge, then either $\psi_n \equiv 0$ in a neighbourhood of $x$ or $\psi_n'(x) \neq 0$. But since $\psi_n$ is a solution of the equation $-u'' + (q-\lambda_n)u = 0$ in an open interval about $x$ and $q \in L^1$, this is an immediate consequence of known maximum principles for solutions of such Schrödinger equations, see \cite{BerSmaTes15}.
\end{proof}

\section{A stronger Pleijel's Theorem for the Laplacian with natural vertex conditions}\label{sec:plei-lapl-nat}

In the particular case of the free Laplacian with natural conditions at all vertices, we can say somewhat more. The following, our second main result, is a complement to the main result in~\cite{GnuSmiWeb04}, whose scope we also extend by removing the genericity condition therein.

\begin{theorem}
\label{thm:standard}
In addition to Assumption~\ref{assum:basic}, suppose that $q\equiv 0$, $\VM=0$, $\mw\equiv 1$, $\mV_0=\emptyset$. Then the following assertions hold.
\begin{enumerate}
\item If $\Graph$ does not contain any loops, then the set of edge length vectors in $\R^{|\mE|}_+$ for which, for the corresponding graph with the given topology and these edge lengths, all eigenvalues are simple and $\lim_{n \to \infty} \frac{\nu_n}{n} = 1$, is of the second Baire category (i.e., is a countable intersection of open dense sets).
\item If $\mathcal G$ contains a loop of length $\ell$, then $\frac{\ell}{|\mathcal G|}$ is a point of accumulation of $\frac{\nu_n}{n}$. In particular, the lower estimate of \eqref{eq:estimates} is sharp whenever $\ell_{\min}$ is realised by a loop.
\item If all edge lengths of $\Graph$ are rationally dependent, then there exists an orthonormal basis of eigenfunctions such that $\limsup_{n \to \infty} \frac{\nu_n}{n} = 1$. If $\Graph$ contains a cycle, and is not a loop, then the basis may be chosen so that  additionally  $\liminf_{n \to \infty} \frac{\nu_n}{n} < 1$ holds.
\end{enumerate}
\end{theorem}

Put differently, in the case of natural vertex conditions and no potential, ``almost all'' graphs (in the usual sense of holding generically and being loop-free) have all eigenvalues simple, and satisfy $\lim_{n\to\infty} \frac{\nu_n}{n} = 1$; however, {at least for non-trees,} if the edge lengths are rationally dependent then this is never the case. Part (2) has an obvious consequence which is nevertheless worth stating explicitly: given any $\varepsilon > 0$ there exists a graph $\mathcal G$ such that for this graph $\liminf_{n \to \infty} \frac{\nu_n}{n} < \varepsilon$.

\begin{remark}
Parts (1) and (2) of Theorem~\ref{thm:standard} also hold, with essentially identical proofs, if any mix of delta couplings and Dirichlet conditions is imposed at some vertices, although for (1) we still need a certain additional genericity assumption (coming from \cite[Theorem~3.6]{BerLiu17}) on the delta conditions. {We expect part (3) to hold for many tree graphs as well, although here the situation is more complicated, as Example~\ref{ex:3-star-3} shows.}
\end{remark}

Actually, we expect that on \emph{any} graph $\Graph$ there exists a choice of (natural Laplacian) eigenfunctions for which we have $\limsup_{n \to \infty} \frac{\nu_n}{n} = 1$. This would be an immediate consequence of the following conjecture together with Proposition~\ref{prop:accumulation}. 
\begin{conjecture}
\label{conj:no}
Let, as usual, $\Graph$ be a compact, connected metric graph. Then there exists a choice of eigenfunctions $\psi_n$ for the Laplacian with natural vertex conditions on $\Graph$ such that the eigenfunctions form an orthonormal basis of $L^2(\Graph)$ and, for a subsequence ${n_k} \in \N$, no eigenfunction $\psi_{n_k}$ vanishes identically on any edge of $\Graph$. In order words, for this choice of eigenfunctions,
\begin{displaymath}
	1 \in \acc \left\{ \frac{|\{\psi_n \neq 0\}|}{|\mathcal G|} : n \in \N \right\}.
\end{displaymath}
\end{conjecture}

It follows from parts (1) and (2) of Theorem~\ref{thm:mainresult} that the conjecture is true generically; it is also true in the case where all edge lengths are rationally dependent, by (3). A counterexample would hence require a graph to have at least two rationally independent edge lengths. Additionally, topological constraints exist, too: it follows from~\cite[Lemma~2.7 and Corollary~2.8]{Ser20} that so-called \textit{lasso trees} (i.e., graphs that can be constructed by attaching at most one loop to any leaf of a tree) cannot be counterexamples, either. Before giving the proof of Theorem~\ref{thm:standard}, we will give a simple example which shows that the sequence $\frac{\nu_n}{n}$, and even its set of points of accumulation, can depend on the choice of the basis of eigenfunctions $\psi_n$.

\begin{example}
\label{ex:4-star-choice}
Consider the equilateral $4$-star; more precisely, we take $\Graph$ to consist of four edges $\me_1, \ldots, \me_4$, each of length $1$ and identified with the interval $[0,1]$, joined at a common vertex of degree four (identified with $0$ on each edge) and with the other four vertices each being of degree one. Then there are two families of eigenfunctions (and corresponding eigenvalues):
\begin{itemize}
\item Eigenfunctions which are invariant under permutation of the edges; up to scalar multiples these are of the form $\varphi_k (x) = \cos (\pi k x)$, $k \in \N$, on each edge $\me_j \simeq [0,1]$, with corresponding eigenvalues $\pi^2k^2$, each of which has multiplicity one.
\item Eigenfunctions which vanish at the central vertex: the corresponding eigenvalues, $\pi^2(k-\frac{1}{2})^2$, $k \in \N$, all have multiplicity three. Any function $\phi$ in the eigenspace has the form $c_{j} \sin (\pi(k-\frac{1}{2})x)$ on each edge $\me_j$, where the coefficients $c_{j} = c_{j}(\phi) \in \R$ are chosen in such a way that the Kirchhoff condition is satisfied at the vertex.
\end{itemize}
We present two different choices for the $c_{j}$, which give rise to two different families of orthogonal bases with different nodal counts. To keep the presentation more compact and easier to read, we present these choices in table form:\\
\begin{center}
\begin{tabular}{l|c|c|c|c}
& $c_1$ & $c_2$ & $c_3$ & $c_4$\\  \hline
$\phi_1$ & 1 & -1 & 0 & 0  \\ \hline
$\phi_2$ & 0 & 0 & 1 & -1 \\ \hline
$\phi_3$ & 1 & 1 & -1 & -1 \\
\end{tabular}
\hspace{10pt}
\begin{tabular}{l|c|c|c|c}
& $c_1$ & $c_2$ & $c_3$ & $c_4$\\  \hline
$\phi_1$ & 1 & -1 & 0 & 0  \\ \hline
$\phi_2$ & 1 & 1 & -2 & 0 \\ \hline
$\phi_3$ & 1 & 1 & 1 & -3 \\
\end{tabular}
\end{center}
Thus, for example, in the second case, for each $k \in \N$ there is an eigenfunction $\phi_3 = \phi_3(k)$ which takes the form $\sin (2\pi(k-1)x)$ on each of $\me_1$, $\me_2$ and $\me_3$,  and $-3\sin (2\pi(k-1)x)$ on $\me_4$. The orthogonality of $\phi_1,\phi_2,\phi_3$ within each family is easy to check, as we simply require that the respective row vectors have inner product zero with each other; while the Kirchhoff condition is satisfied as long as the sum of the entries in each vector is zero. (The eigenfunctions will not have norm one, but this is obviously just a question of rescaling.) Now in the first family, there are two eigenfunctions each supported on two different edges and one supported on all four; in the second family, the second eigenfunction is supported on three edges rather than two. It follows from Proposition~\ref{prop:accumulation} (also taking into account the nature of the eigenfunctions not vanishing on the central vertex) that in the first case the set of points of accumulation of the sequence $\frac{\nu_n}{n}$ is $\{\frac{1}{2}, 1\}$ and in the second case it is $\{\frac{1}{2}, \frac{3}{4}, 1\}$.
\end{example}

While part (3) of Theorem \ref{thm:standard} states that, for all graphs containing a cycle and having rationally dependent edge lengths, there  exists an orthonormal basis of eigenfunctions such that \(\liminf_{n\rightarrow\infty}\frac{\nu_n}{n}<1\), this is not necessarily true if the graph is a tree: indeed, the following example shows that there are trees with rationally dependent edge lengths where any eigenfunction of the Laplacian with natural conditions is supported on the whole tree, which in turn yields, by Proposition \ref{prop:accumulation}, that \(\lim_{n\rightarrow\infty}\frac{\nu_n}{n}=1\) holds for any orthonormal basis of eigenfunctions.

\begin{example}
\label{ex:3-star-3}
Consider the \(3\)-star \(\mathcal G\) consisting of three edges \(\me_1,\me_2,\me_3\) of edge lengths \(\ell_1,\ell_2,\ell_3\) respectively. An eigenfunction \(\varphi=(\varphi_1,\varphi_2,\varphi_3)\) corresponding to some eigenvalue \(\lambda>0\) is of the form \(\varphi_j(x)=c_j\cos(\sqrt{\lambda}x)\) on the edge \(\me_j\simeq [0,\ell_j]\) where \(\ell_j\) corresponds to the centre vertex of the star. If \(\varphi\) vanished on some edge \(\me_i\), we would obtain \(c_i=0\) and \(c_j\neq 0\) for \(j\neq i\). Then continuity in the centre vertex yields \(0=\varphi_j(\ell_j)\) for \(j\neq i\) and, thus \(0=\cos(\ell_j\sqrt{\lambda})\). Therefore there is some \(m_j\in\mathbb N\) such that \(\ell_j\sqrt{\lambda}=\pi(m_j-\frac{1}{2})\). This yields
\begin{equation}\label{eq:contradiction-example-fully-supp-ef}
		\frac{\ell_k}{\ell_j}(2m_k-1)=2m_j-1
\end{equation}
for \(k,j\neq i\). Now we choose \(\ell_1=1\), \(\ell_2\) and \(\ell_3=4\). Suppose without loss of generality that \(\ell_k> \ell_j\) in \eqref{eq:contradiction-example-fully-supp-ef}. Then, with our choice of the edge lengths, \eqref{eq:contradiction-example-fully-supp-ef} clearly leads to a contradiction, since the left-hand side  is an even integer, whereas the right-hand side is odd. Therefore all eigenfunctions on the \(3\)-star with edge lengths \(1\), \(2\) and \(4\) must be supported on the whole graph.
\end{example}

\begin{proof}[Proof of Theorem~\ref{thm:standard}]
(1) Under the stated conditions, the set of edge length vectors for which all eigenvalues of the corresponding graph are simple and none of the eigenfunctions vanish at any vertices, is of the second Baire category in $\R^{|\mE|}_+$: this is the principal result of \cite{BerLiu17} (see Theorem~3.6 and Remark~3.7 there). It follows in particular that no eigenfunction vanishes identically on any edge (cf.\ also Lemma~\ref{lem:values-of-psin} and its proof); in particular, $|\{\psi_n \neq 0\}| = |\Graph|$ for all $n \in \N$. Now the statement is an immediate consequence of Proposition~\ref{prop:accumulation}.

(2) If $\mathcal G$ is just a loop of length $\ell$, then no eigenfunction vanishes identically on any set of positive measure, so we are again in the situation where $|\{\psi_n \neq 0\}| = |\Graph|$ for all $n \in \N$ and hence $\frac{\ell}{|\mathcal G|} = 1$ is indeed a (in fact the only) point of accumulation of $\frac{\nu_n}{n}$. So suppose $\me$ is a loop attached to $\Graph \setminus \me$ at a vertex $\mv$; then there exists a sequence of eigenvalues supported on $\me$ (more precisely, making the identification $\me \simeq (0,\ell_\me)$, where both $0$ and $\ell_\me$ correspond to $\mv$, the eigenvalues associated with the family of eigenfunctions $\psi_j (x) = \sin(\frac{j \pi x}{\ell_e})$ on $\me$ extended by zero on $\Graph \setminus \me$ have this property). The assertion is now, again, an immediate consequence of Proposition~\ref{prop:accumulation}.

(3) By inserting dummy vertices as necessary, we may assume that the graph is in fact equilateral; after rescaling if necessary, we may also assume without loss of generality that each edge has length $1$. The following proof is essentially based on the possibility of considering all eigenfunctions as linear combinations of full frequency eigenfunctions on each edge, as in~\cite{Bel85}.

We first show that for the correct choice of eigenfunctions $\limsup_{n \to \infty} \frac{\nu_n}{n} = 1$. To do so, for each $k \in \N$ we construct a function $\varphi_k \in H^1(\Graph) \setminus \{0\}$ by setting $\varphi_k(x)=\cos(2\pi k x)$ on each edge $\me \simeq [0,1]$. Note that this function is continuous as it takes the value $1$ at every vertex. We also see that $\varphi_k$ satisfies the Kirchhoff condition at every vertex, as its normal derivative on each edge pointing into any vertex is always zero by construction. Hence it must be an eigenfunction of the Laplacian with natural vertex conditions, and eigenvalue $4\pi^2 k^2$. This gives us an infinite sequence of eigenfunctions $\varphi_k = \psi_{n_k}$ (for some $n_k \in \N$), each of which is supported on the whole of $|\Graph|$. That $\limsup_{n \to \infty} \frac{\nu_n}{n} = 1$ now follows from Theorem~\ref{thm:mainresult} and Proposition~\ref{prop:accumulation}.

It remains to show that we can find another sequence of eigenfunctions, orthogonal to the $\varphi_k$, none of which are supported on the whole graph if  $\Graph$ contains a cycle: this will show that $\liminf_{n\to \infty}\frac{\nu_n}{n}<1$. By assumption the cycle cannot exhaust $\Graph$. We construct eigenfunctions $\phi_k$, $k \in \N$, by setting $\phi_k (x) = \sin (2\pi k x)$ on each edge $\me \simeq [0,1]$ belonging to the cycle and $0$ on the rest of the graph. That $\phi_k$ is indeed an eigenfunction can be seen immediately, as can the fact that $\frac{|\phi_k \neq 0|}{|\Graph|} \in (0,1)$ is a constant independent of $k$. Moreover, clearly the $\phi_k$ and the $\varphi_j$ are always orthogonal to each other; hence there exists a choice of eigenfunctions for $\Graph$ containing both sequences (up to renormalisation to ensure the eigenfunctions have norm $1$).
\end{proof}

\section{Pleijel's theorem for the $p$-Laplacian}
\label{sec:pleijel-p}

In this last section we are going to turn to a different class of operators. The $p$-Laplacian on metric graphs can be generally introducing by considering the Fréchet differentiable energy functional
\[
{\mathfrak E}_p:
u\mapsto \int_{\mathcal G} |u'|^p \dx,\qquad u\in D({\mathfrak E}_p):= W^{1,p}(\Graph),
\]
and taking its Fréchet derivative in the real Hilbert space $L^2(\mathcal G)$; this returns natural vertex conditions, i.e., continuity across the vertices along with a nonlinear analogue of Kirchhoff's condition. Unlike in the linear case of $p=2$, different notions of eigenvalues for the $p$-Laplacian may \textit{a priori} coexist, see Appendix~\ref{sec:p-weyl}, with \textit{Carathéodory} eigenvalues being more general than \textit{variational} ones. Given a general compact metric graph, it seems to be unknown how large the the set of Carathéodory eigenvalues of this operator is, but its subset that is most relevant for our purposes -- the set of variational eigenvalues -- is certainly countably infinite; such variational eigenvalues can be characterised by the Ljusternik--Schnirelmann principle, a nonlinear counterpart of the linear min-max principle.

Here we will denote by $(\lambda_{n,p}(\Graph))_{n\in\mathbb N}$ the sequence of variational eigenvalues, along with a sequence of associated (Carathéodory) eigenfunctions $(\psi_{n,p})_{n\in\mathbb N}$, which we fix throughout; each eigenfunction has $\nu_{n,p}$ corresponding nodal domains $\mathcal G_1, \ldots, \mathcal G_{\nu_{n,p}}$.

Actually, in view of the nonlinear versions of the Beurling--Dény conditions in~\cite{CipGri03}, as in \eqref{eq:form-2}, different vertex conditions inducing (nonlinear) positive semigroups  can be obtained upon considering the above energy on spaces of the form $W^{1,p}_w(\Graph)$ and/or adding boundary terms; we expect our results to continue to hold for these. However, owing to a lack of background theory available for such nonlinear operators on metric graphs, we will not pursue such generalisations here.

In this section we will always impose the following
\begin{assumption}\label{assum:basic-p}
$\mathcal G$ is a compact, connected metric graph with underlying combinatorial graph $\mG=(\mV,\mE)$ and edge lengths $\ell_\me$, $\me\in\mE$; we set $\ell_{\min} := \min_{\me \in \mE} \ell_\me$. We also fix $p \in (1,\infty)$ and let $q = \frac{p}{p-1}$ be its Hölder conjugate.
\end{assumption}

Our third main result, a version of Pleijel's theorem for the $p$-Laplacian with natural vertex conditions, is a direct analogue of Theorem~\ref{thm:mainresult}.

\begin{theorem}
\label{thm:pleijel-p}
Under Assumption~\ref{assum:basic-p}, and with the notation on the nodal count introduced above, we have
\begin{equation}
\label{eq:accumulation-p}
	\acc \left\{ \tfrac{\nu_{n,p}}{n} : n \in \N \right\} \subset \left\{ \frac{\sum_{\me \in \mE_0} \ell_\me}{|\mathcal G|} : \mE \supset \mE_0
	\text{ is a nonempty set of edges} \right\}.
\end{equation}
In particular, $\acc \left\{ \tfrac{\nu_{n,p}}{n} : n \in \N \right\}$ is a finite set, and
\begin{equation}
\label{eq:estimates-p}
	0 < \frac{\ell_{\min} }{|\mathcal{G}|} \leq \liminf_{n \to \infty} \frac{\nu_{n,p}}{n} \leq \limsup_{n \to \infty} \frac{\nu_{n,p}}{n} \leq 1,
\end{equation}
where $\ell_{\min} := \min \{ \ell_\me : \me \in \mE \}$.
\end{theorem}

We also observe that Proposition~\ref{prop:accumulation} holds verbatim with $\nu_{n,p}$ and $\psi_{n,p}$ in place of $\nu_n$ and $\psi_n$, respectively. The proof of Theorem~\ref{thm:pleijel-p} (and Proposition~\ref{prop:accumulation} in this case) follows exactly the same lines as above.

In this case, we give a short proof of the Weyl asymptotics for $\lambda_{n,p}(\Graph)$ in the appendix (see Theorem~\ref{thm:p-weyl}), as it does not previously seem to have been established for the $p$-Laplacian on metric graphs. We next state $p$-versions of unique continuation (cf.\ Lemma~\ref{lem:values-of-psin}), the fact that $\lambda_{n,p}$ is the first Dirichlet eigenvalue restricted to each nodal domain of $\psi_{n,j}$ (cf.\ Lemma~\ref{lem:first-eigenvalue}) and a basic upper bound on the first Dirichlet eigenvalue (cf.\ Proposition~\ref{prop:lambda1-est}), respectively.

The following lemma on unique continuation is actually valid for any vertex conditions enforced in the (real) Sobolev space $W^{1,p}(\mathcal G)$, the domain of ${\mathfrak E}_p$, since they necessarily result in real eigenvalues and eigenfunctions.

\begin{lemma}[Possible values of $|\{\psi_{n,p} \neq 0\}|$]
\label{lem:values-of-p-psin}
Under Assumption~\ref{assum:basic-p},
\begin{equation}
\label{eq:values-of-p-psin}
	\Big\{ |\{\psi_{n,p} \neq 0\}| : n \in \N \Big\} \subset \left\{ \sum_{\me \in \mE_0} \ell_\me : \mE \supset \mE_0 
	\text{ is a nonempty set of edges} \right\}.
\end{equation}
\end{lemma}

\begin{proof}
This follows immediately from the assertion that if $\psi_{n,p} (x) = 0$ for some $x$ in the interior of an edge $\me$, then either $\psi_{n,p}$ changes sign in any open neighbourhood of $x$, or $\psi_{n,p}$ vanishes identically on that edge. Suppose that $\psi_{n,p} (x) = 0$ at some interior point $x \in \me$, and that $\psi_{n,p}$ does \emph{not} change sign at $x$. Then by the smoothness properties of $\psi_{n,p}$ stated in Lemma~\ref{lem:car}, we also have $\psi_{n,p}'(x) = 0$. That is, $\psi_{n,p}$ is a solution of
\begin{displaymath}
\begin{split}
	u' &=|v|^\frac{1}{p-1}\sign v \\
	v' &=-\lambda |u|^{p-1}\sign u
\end{split}\qquad \hbox{in a neighbourhood of }x
\end{displaymath}
with boundary conditions
\begin{displaymath}
	u(x) = v(x) =0.
\end{displaymath}
By \cite[Theorem~3.1]{LanEdm11}, this equation has exactly one smooth solution, which in this case is clearly the zero function. Hence $\psi_{n,p}$ vanishes identically in a neighbourhood of $x$ and so, extending the argument, on the whole metric edge $\me\simeq (0,\ell_\me)$.
\end{proof}

\begin{lemma}
\label{lem:p-first-eigenvalue}
Under Assumption~\ref{assum:basic-p}, for all $n\in\mathbb N$
\begin{displaymath}
	\lambda_{n,p} (\mathcal G) = \lambda_1 (\mathcal G_j),
\end{displaymath}
where the latter is the smallest variational eigenvalue of the $p$-Laplacian on $\mathcal G_j$ with Dirichlet conditions at all the boundary points of $\mathcal G_j$ corresponding to zeros of $\psi_{n,p}$ and natural conditions at all other vertices of $\mathcal G_j$.
\end{lemma}

\begin{proof}
In analogy with~\eqref{eq:h10}, denote by $W^{1,p}_0 (\mathcal G_j; \partial\mathcal G_j)$ the domain of the functional associated with the eigenvalue problem on $\mathcal G_j$ as described in the assertion; then by choice of $\mathcal G_j$, $\psi_{n,p}|_{\mathcal G_j} \in W^{1,p}_0 (\mathcal G_j; \partial\mathcal G_j)$. As usual, in a slight abuse of notation we will identify $W^{1,p}_0 (\mathcal G_j; \partial\mathcal G_j)$ with a closed subspace of $W^{1,p} (\mathcal G)$ and in particular simply write $\psi_{n,p} \in W^{1,p}_0 (\mathcal G_j; \partial\mathcal G_j)$. We start by observing that $\psi_{n,p}$ is clearly an eigenfunction on $\mathcal G_j$, for the eigenvalue $\lambda_{n,p}(\mathcal G)$, as follows from the fact that
\begin{displaymath}
	\int_{\mathcal G} |\psi_{n,p}'(x)|^{p-2}\psi_{n,p}'(x)\varphi'(x)\,\textrm{d}x = \lambda_{n,p}(\mathcal G) \int_{\mathcal G} |\psi_{n,p}(x)|^{p-2}\psi_{n,p}(x)\varphi(x)\,\textrm{d}x
\end{displaymath}
for all $\varphi \in W^{1,p} (\mathcal G)$ and hence, in particular, for all $\varphi \in W^{1,p}_0 (\mathcal G_j; \partial\mathcal G_j)$.
Moreover, $\psi_{n,p}$ is either strictly positive or strictly negative in (the connected set) $\mathcal G_j \setminus \partial\mathcal G_j$, as is an immediate consequence of the definition of nodal domains. The proof of \cite[Theorem~1.1]{KawLin06} may now be repeated verbatim to show that $\lambda_{n,p}(\mathcal G)$ is in fact the first eigenvalue of the $p$-Laplacian on $\mathcal G_j$ with the desired vertex conditions.
\end{proof}

The following upper bound was proved in \cite[Theorem~3.8]{DelRos16}. Again, this bound extends to the lowest variational eigenvalue of all realisations of the $p$-Laplacian induced by the functional $\mathfrak E_p$ defined on a superset of $W^{1,p}_0(\Graph)$.

\begin{lemma}
\label{lem:p-upper-bound}
Under Assumption~\ref{assum:basic-p}, let $\mV_0$ be a (finite) non-empty set of points of $\Graph$, such that $\Graph \setminus \mV_0$ is connected, and, for $p \in (1,\infty)$, let $\lambda_{1,p} (\Graph; \mV_0)$ be the first eigenvalue of the $p$-Laplacian with Dirichlet conditions at $\mV_0$ and natural conditions at all other vertices. Then
\begin{equation}
\label{eq:p-upper-bound}
	\lambda_{1,p} (\Graph; \mV_0) \leq \frac{p}{q} \left(\frac{\pi_p|\mE|}{|\Graph|}\right)^p.
\end{equation}
\end{lemma}

The final auxiliary result we need is an analogue of Lemma~\ref{lem:nodal-size-est}, an estimate from above on the size of the nodal domains (equivalently, a lower bound on $\lambda_{n,p}$), which is itself a direct consequence of the preceding two lemmata. This establishes in particular (together with Lemma~\ref{lem:values-of-p-psin}) that the number of nodal domains does in fact diverge to infinity as $n \to \infty$.

\begin{lemma}
\label{lem:p-nodal-size-est}
Fix $n \in \N$ and let $\Graph_1,\ldots,\Graph_{\nu_{n,p}}$ be the nodal domains of $\psi_{n,p}$. Then for all $j=1,\ldots,\nu_{n,p}$ we have
\begin{equation}
\label{eq:p-nodal-size-est}
	|\Graph_j| \leq \frac{2\pi_p|\mE|p^{1/p}}{(q\lambda_{n,p}(\Graph))^{1/p}}.
\end{equation}
In particular, if $n\in\N$ is large enough, specifically, if $\lambda_{n,p}(\Graph) > \frac{p}{q}\left(\frac{2\pi_p|\mE|}{|\Graph_j|}\right)^p$, then no nodal domain can contain more than one vertex.
\end{lemma}

\begin{proof}
Fix a nodal domain $\Graph_j$, then since $\Graph_j$ cannot have more than $2|\mE|$ edges, by Lemma~\ref{lem:p-first-eigenvalue} and Lemma~\ref{lem:p-upper-bound}, the latter applied to $\Graph_j$, we have
\begin{displaymath}
	\lambda_{n,p} (\Graph) = \lambda_{1,p}(\Graph_j;\partial\Graph_j) [ = \lambda_{1,p}^D (\Graph_j)] \leq \frac{p}{q} \left(\frac{\pi_p|\mE|}{|\Graph|}\right)^p.
\end{displaymath}
Rearranging yields \eqref{eq:p-nodal-size-est}. The other assertion is clear.
\end{proof}

We can now formulate a version of the central Lemma~\ref{lem:nu-lambda} for the $p$-Laplacian.

\begin{lemma}
\label{lem:p-nu-lambda}
For all sufficiently large $n \in \N$, we have
\begin{equation}
\label{eq:p-nu-lambda-estimates}
	\frac{|\psi_{n,p} \neq 0|}{\pi_p} \cdot \left(\frac{q\lambda_{n,p}(\Graph)}{p}\right)^{1/p} - (2|\mE|-1)|\mV|
	\leq \nu_{n,p} \leq \frac{|\psi_{n,p} \neq 0|}{\pi_p} \cdot \left(\frac{q\lambda_{n,p}(\Graph)}{p}\right)^{1/p} + |\mV|
\end{equation}
\end{lemma}

Concretely, the condition on $\lambda_{n,p}(\Graph)$ from Lemma~\ref{lem:p-nodal-size-est} is enough to ensure that \eqref{eq:p-nu-lambda-estimates} holds.

\begin{proof}
We suppose $n$ is large enough that there are in fact $|\mV|$ nodal domains containing exactly one vertex of $\Graph$, while the rest contain no vertices; that this is possible is guaranteed by Lemma~\ref{lem:p-nodal-size-est}. Let $\Graph_1, \ldots, \Graph_{\nu_{n,p}}$ be the nodal domains of $\psi_{n,p}$. We assume that $\Graph_1,\ldots,\Graph_{|\mV|}$ each contain a vertex, while the rest do not; then each $\Graph_j$ is an interval with Dirichlet conditions at its endpoints if $j > |\mV|$, and in this case
\begin{displaymath}
	\lambda_{n,p}(\Graph) = \lambda_{1,p}(\Graph_j;\partial\Graph_j) = \frac{p}{q} \left(\frac{\pi_p}{|\Graph_j|}\right)^p,
\end{displaymath}
i.e., $|\Graph_j| = \pi_p \left(\frac{p}{q\lambda_{n,p}(\Graph)}\right)^{1/p}$. Hence, as in the proof of Lemma~\ref{lem:nu-lambda}, using the definition of the nodal domains,
\begin{displaymath}
	|\{\psi_{n,p}\neq 0\}| = \sum_{j=1}^{\nu_{n,p}}|\Graph_j| = \sum_{j=1}^{|\mV|} |\Graph_j| + (\nu_{n,p}-|\mV|)
	\pi_p \left(\frac{p}{q\lambda_{n,p}(\Graph)}\right)^{1/p}
\end{displaymath}
The sum on the right-hand side is non-negative and may be controlled from above using Lemma~\ref{lem:p-nodal-size-est}; this yields
\begin{displaymath}
	(\nu_{n,p}-|\mV|) \pi_p \left(\frac{p}{q\lambda_{n,p}(\Graph)}\right)^{1/p}\leq |\{\psi_{n,p}\neq 0\}|
		\leq (\nu_{n,p}-|\mV|) \pi_p \left(\frac{p}{q\lambda_{n,p}(\Graph)}\right)^{1/p} + 2\pi_p|\mV||\mE|\left(\frac{p}{q\lambda_{n,p}(\Graph)}\right)^{1/p}.
\end{displaymath}
Rearranging yields \eqref{eq:p-nu-lambda-estimates}.
\end{proof}

\begin{proof}[Proof of Theorem~\ref{thm:pleijel-p} and of Proposition~\ref{prop:accumulation} for the $p$-Laplacian]
Upon combining the result of Lemma~\ref{lem:p-nu-lambda} with the Weyl asymptotics of Theorem~\ref{thm:p-weyl}, we obtain
\begin{displaymath}
	\nu_{n,p} = \frac{|\{\psi_{n,p}\neq 0\}|}{|\Graph|} n + o(n) \qquad \text{as } n \to \infty,
\end{displaymath}
which in particular proves Proposition~\ref{prop:accumulation} for the $p$-Laplacian. Lemma~\ref{lem:values-of-p-psin} now yields \eqref{eq:accumulation-p}; the other assertions of Theorem~\ref{thm:pleijel-p} follow immediately.
\end{proof}

\appendix

\section{Weyl's law for the $p$-Laplacian on metric graphs}
\label{sec:p-weyl}

The goal of this section is, firstly, to recall briefly the construction of the variational eigenvalues of the $p$-Laplacian (with natural vertex conditions, that is, continuity and an appropriate $p$-version of the Kirchhoff condition); this is well known on intervals and domains, and nothing changes in the case of metric graphs (see also \cite{DelRos16}); secondly, we will show that the Weyl asymptotics known for the $p$-Laplacian eigenvalues on the interval also holds on metric graphs. This is a simple application of Dirichlet--Neumann bracketing.

We recall that the $n$-th variational eigenvalue of the $p$-Laplacian on a graph $\Graph$ with natural vertex conditions, $p \in (1,\infty)$, may be characterised variationally in terms of the Krasnosel'skii genus. More precisely, analogously to \cite[Section~5]{BinDra03}, see also \cite[Section~3]{DraRob02},  we consider the manifold
\begin{displaymath}
	\mathcal{S} := \left\{ f \in W^{1,p} (\Graph) : \|f\|_{L^p(\Graph)}^p = 1 \right\},
\end{displaymath}
and for a closed, symmetric, non-empty set $\mathcal{A} \subset \mathcal{S}$ its Krasnosel'skii genus $\gamma (\mathcal{A}) \in \N$ by
\begin{displaymath}
	\gamma(\mathcal{A}) := \inf \{k \in \N: \text{ there exists } \Phi: \mathcal{A} \to \mathbb{S}^{k} \text{ continuous and odd} \},
\end{displaymath}
where \(\mathbb S^k\) denotes the unit sphere in \(\mathbb R^k\) for \(k\in\mathbb N\) (or $\gamma (\mathcal{A}) = \infty$ if this infimum is infinite). Finally, for every $n \in \N$ we set $\mathcal{F}_n := \{ \mathcal{A} \subset \mathcal{S}: \gamma (\mathcal{A}) \geq n\}$. 
Then we may define the $n$-th variational eigenvalue $\lambda_{n,p} (\Graph)$ of the $p$-Laplacian on $\Graph$ with natural vertex conditions by
\begin{equation}
\label{eq:p-n}
	\lambda_{n,p} (\Graph) = \inf_{\mathcal{A} \in \mathcal{F}_n} \sup_{f \in \mathcal{A}} \int_{\Graph} |f'(x)|^p\,\textrm{d}x.
\end{equation}
That this does indeed give rise to an infinite sequence of eigenvalues on any compact metric graph $\Graph$ follows from the same argument as the one used in \cite{BinDra03}, see also \cite{DraRob99,DraRob02}. While \textit{a priori} $(\lambda_{n,p} (\Graph))_{n\in\mathbb N}$ is just a sequence of critical points of a certain functional, mimicking the proof of~\cite[Theorem~2.1]{BinRyn08} one can show by known methods that each such \textit{variational eigenvalue} is actually associated with an eigenfunction in the following weak sense.
\begin{lemma}\label{lem:car}
For each $n\in \mathbb N$ there exists a (so-called \textit{Carathéodory}) eigenfunction associated with $\lambda = \lambda_{n,p}(\Graph)$, i.e., a non-zero solution $\psi_{n,p}=u$ of the system
\[
\begin{split}
u' &=|v|^\frac{1}{p-1}\sign v \\
v' &=-\lambda |u|^{p-1}\sign u.
\end{split}
\]
such that $u$ and $v$ satisfy the continuity and Kirchhoff-type vertex conditions, respectively. In particular, $\psi_{n,p}$ is a real, absolutely continuous function, and so is $|\psi_{n,p}|^{p-1}\sign \psi_{n,p}$.
\end{lemma}
In particular, and with the terminology of~\cite{BinRyn08}: like on intervals with Dirichlet or Neumann boundary conditions, each variational eigenvalue is a Carathéodory eigenvalue, too.

We also define the corresponding eigenvalues in the case that all vertices of $\Graph$ are equipped with either a Dirichlet or a Neumann condition, in which case $\Graph$ decomposes into the disjoint union of $|\mE|$ edges, or intervals; this obviously includes the case $|\mE|=1$ where $\Graph$ is just a (bounded) interval itself. We define the natural analogues of $\mathcal{S}$, namely
\begin{displaymath}
\begin{aligned}
	\mathcal{S}^D &:= \left\{ f \in W^{1,p}_0 (\Graph) : \|f\|_{L^p(\Graph)}^p = 1 \right\},\\
	\mathcal{S}^N &:= \left\{ f \in \bigoplus_{\me\in\mE} W^{1,p} (0,\ell_\me) : \|f\|_{L^p(\Graph)}^p = 1 \right\},
\end{aligned}
\end{displaymath}
where $W^{1,p}_0 (\Graph)$ is, analogously to $H^1_0 (\Graph):=H^1(\mathcal G; \mV)$ in~\eqref{eq:h10}, the space of all functions in $W^{1,p}(\Graph)$ vanishing at \textit{all} vertices, and $\bigoplus_{\me\in\mE} W^{1,p} (0,\ell_\me)$ is to be identified with a superset of $W^{1,p} (\Graph)$ in the obvious way. Then, defining the Krasnosel'skii genus in the same way as above, and finally
\begin{displaymath}
	\mathcal{F}_n^{D,N} := \{ \mathcal{A} \subset \mathcal{S}^{D,N}: \gamma (\mathcal{A}) \geq n\},
\end{displaymath}
we define the respective $n$-th variational eigenvalues by
\begin{equation}
\label{eq:p-n-dn}
	\lambda_{n,p}^{D,N} (\Graph) = \inf_{\mathcal{A} \in \mathcal{F}_n^{D,N}} \sup_{f \in \mathcal{A}} \int_{\Graph} |f'(x)|^p\,\textrm{d}x.
\end{equation}
Again, it is easy to see that in both cases there is a sequence of eigenvalues; this is proved explicitly in \cite[Theorems 3.3 and 3.4]{LanEdm11} for the $p$-Laplacian on intervals (but it makes no difference if we consider a disjoint union of intervals). We may also consider eigenvalues $\lambda_{n,p}^D (\Graph; \mV_0)$ with a Dirichlet condition imposed at some subset $\mV_0$ of the vertices and natural conditions at the rest; all the definitions are analogous and we do not go into details.

The following Dirichlet--Neumann bracketing principle is an immediate consequence of the respective eigenvalue definitions.

\begin{lemma}
\label{lem:p-bracketing}
Fix $p \in (1,\infty)$ and $\Graph$. With the notation introduced above, we have
\begin{displaymath}
	\lambda_{n,p}^N (\Graph) \leq \lambda_{n,p} (\Graph) \leq \lambda_{n,p}^D (\Graph)
\end{displaymath}
for all $n \geq 1$.
\end{lemma}

\begin{proof}
We observe that $\mathcal{S}^D \subset \mathcal{S} \subset \mathcal{S}^N$, whence $\mathcal{F}_n^D \subset \mathcal{F}_n \subset \mathcal{F}_n^N$. The statement is now an immediate consequence of the characterisations \eqref{eq:p-n} and \eqref{eq:p-n-dn}.
\end{proof}

\begin{theorem}[Weyl asymptotics]
\label{thm:p-weyl}
Fix $p \in (1,\infty)$ and suppose the graph $\Graph$ has total length $|\Graph|$. Then the $n$-th variational eigenvalue $\lambda_n (\Graph)$ satisfies
\begin{equation}
\label{eq:p-weyl}
	\lambda_{n,p} (\Graph) = (p-1)\left(\frac{\pi_p}{|\Graph|}\right)^p n^p + o(n^p) \qquad \text{as } n \to \infty,
\end{equation}
where we recall $\pi_p = \frac{2\pi}{p \sin(\frac{\pi}{p})}$.
\end{theorem}

A corresponding Weyl asymptotics for the Dirichlet $p$-Laplacian on general domains in $\R^n$ was established only very recently, see \cite{Maz19}.

\begin{proof}
We first observe that the Weyl asymptotics \eqref{eq:p-weyl} holds for the $p$-Laplacian on an interval with both Dirichlet and Neumann boundary conditions (see \cite[Theorems~3.3 and 3.4]{LanEdm11}. Hence it also holds in the case that $\Graph$ is a disjoint collection of intervals, equivalently, for any graph $\Graph$ it holds for $\lambda_{n,p}^N (\Graph)$ and $\lambda_{n,p}^D (\Graph)$. The conclusion of the theorem now follows immediately from Lemma~\ref{lem:p-bracketing}.
\end{proof}

\section{An estimate on the first eigenvalue of general Schrödinger operators}
\label{app:lambda1-est}

In this appendix we give an estimate on the first eigenvalue $\lambda_1 (\Graph)$ of any Schrödinger operator $\Op = \Op (q, \VM, w, \mV_0)$ of the form introduced in Section~\ref{sec:gen-setting}, on any compact metric graph. Estimates of this level seem to be new at this level of generality and may be of some independent interest, although there is considerable room for improvement. In practice we will take Assumption~\ref{assum:basic}; however, the statements and proofs are all valid for general $q \in L^1 (\Graph)$, not necessarily positive, in which case the norm $\|q\|_1$ may be replaced by $\|q_+\|_1$, the norm of the positive part of $q$ (this is a trivial consequence of the variational characterisation of $\lambda_1$).

\begin{proposition}
\label{prop:lambda1-est}
Keeping the notation of Sections~\ref{sec:gen-setting} and~\ref{sec:plei-schr}, under Assumption~\ref{assum:basic} we have
\begin{equation}\label{eq:pseudoKKMM}
	\lambda_1 (\Graph) \leq \left(\frac{\pi|\mE|}{|\Graph|} + \|q\|_1\right)^2 - \|q\|_1^2.
\end{equation}
\end{proposition}

Note that $\frac{|\Graph|}{|\mE|}$ is exactly the average edge length of $\Graph$.

\begin{proof}
We first observe that the inequality
\begin{equation}
\label{eq:sobolev-embedding}
	\|f\|_\infty^2 \leq 2 \|f'\|_2\|f\|_2
\end{equation}
is valid for all $f \in H^1_0 (\Graph)$: indeed, fixing any edge $\me$, identified with the interval $[0,\ell_{\me}]$, and any $x \in (0,\ell_{\me})$, by the fundamental theorem of calculus and the Cauchy--Schwarz inequality, since $f(0)=0$ we have
\begin{displaymath}
	|f(x)|^2 = \int_0^x \big(|f(t)|^2\big)'\,\textrm{d}t \leq 2\int_0^x |f'(t)||f(t)|\,\textrm{d}t \leq 2\|f'\|_2\|f\|_2.
\end{displaymath}
Now suppose that $f \in H^1_0 (\mathcal G)$ satisfies $\|f\|_2 = 1$, then by \eqref{eq:sobolev-embedding}
\begin{displaymath}
	\lambda_1 (\Graph) \leq \int_{\mathcal G} |f'(x)|^2+q(x)|f(x)|^2\,\textrm{d}x \leq \|f'\|_2^2 + 2\|q\|_1 \|f'\|_2.
\end{displaymath}
Taking the infimum over all such functions $f$ yields
\begin{displaymath}
	\lambda_1 (\Graph) \leq \lambda_1^D (0) + 2\lambda_1^D(0)^{1/2}\|q\|_1,
\end{displaymath}
where $\lambda_1^D(0)$ is the Dirichlet Laplacian on $\mathcal G$, i.e., with zero potential and Dirichlet conditions at all vertices of $\mathcal G$ (that is, the Dirichlet Laplacian on the collection of $|\mE|$ disjoint intervals comprising the edges of $\mathcal G$). Now at least one edge of $\mathcal G$ has length at least $|\mathcal G|/|\mE|$; and so $\lambda_1^D (0) \leq \pi^2|\mE|^2/|\mathcal G|^2$. This yields \eqref{eq:pseudoKKMM}.
\end{proof}

\bibliographystyle{plain}

\end{document}